\documentclass[12pt, reqno]{amsart}
\usepackage{amsmath, amsthm, amscd, amsfonts, amssymb, graphicx, color}
\usepackage[bookmarksnumbered, colorlinks, plainpages]{hyperref}
\newtheorem{theorem}{Theorem}[section]
\newtheorem{lemma}[theorem]{Lemma}
\newtheorem{proposition}[theorem]{Proposition}
\newtheorem{corollary}[theorem]{Corollary}
\theoremstyle{definition}
\newtheorem{definition}[theorem]{Definition}
\newtheorem{example}[theorem]{Example}

\theoremstyle{remark}
\newtheorem{remark}[theorem]{Remark}
\numberwithin{equation}{section}

\begin{document}
\setcounter{page}{1}

\title[The logarithmic functional mean]{The logarithmic mean of two convex functionals}

\author[M. Ra\"{\i}ssouli and S. Furuichi]{Mustapha Ra\"{\i}ssouli and Shigeru Furuichi$^{*}$}

\address{Mustapha Ra\"{\i}ssouli: Department of Mathematics, Science Faculty, Taibah University, Al Madinah Al Munawwarah, P.O.Box 30097, Zip Code 41477, Saudi Arabia; Department of Mathematics, Science Faculty, Moulay Ismail University, Meknes, Morocco, e-mail: raissouli.mustapha@gmail.com}

\address{$^{*}$ Corresponding author: Shigeru Furuichi, Department of Information Science, College of Humanities and Sciences, Nihon University, 3-25-40, Sakurajyousui, Setagaya-ku, Tokyo, 156-8550, Japan, e-mail: furuichi@chs.nihon-u.ac.jp}



\date{Received: xxxxxx; Revised: yyyyyy; Accepted: zzzzzz.}

\begin{abstract}
The purpose of this paper is to introduce the logarithmic mean of two convex functionals that extends the logarithmic mean of two positive operators. Some inequalities involving this functional mean are discussed as well. The operator versions of the functional theoretical results obtained here are immediately deduced without referring to the theory of operator means.
\end{abstract}

\maketitle

\section{\bf Introduction}

The mean theory arises in various contexts and has recently extensive  developments and various applications. It attracts many mathematicians by its interesting inequalities and nice properties. See \cite{FM2020} for recent advances in mathematical inequalities. The mean theory was at the first time introduced for positive real numbers for over the last centuries, \cite{BMV}. Afterwards, it has been extended from positive real numbers to positive operator arguments, see \cite{KA,NUC} for instance.

For over the last few years, many operator means have been extended from the case where the variables are positive operators to the case where the variables are convex functionals, see \cite{ATR,FUJ,RAB,RAC,RAI1,RAI2,RAF}. Such functional extensions were investigated in the sense that if $m(A,B)$ is an operator mean between two positive linear operator $A$ and $B$ acting on a complex Hilbert space $H$, then the extension of $m(A,B)$ when $A$ and $B$ are replaced by two convex functionals $f$ and $g$, respectively, is a functional ${\mathcal F}(f,g)$ that satisfies the following connection-relationship
$${\mathcal F}({\mathcal Q}_A,{\mathcal Q}_B)={\mathcal Q}_{m(A,B)}.$$
Here the notation ${\mathcal Q}_A$ refers to the convex quadratic function generated by the positive linear operator $A$ i.e. ${\mathcal Q}_A(x)=(1/2)\langle Ax,x\rangle$ for all $x\in H$.

This functional approach, which was investigated under a convex point of view, stems its importance in different facts. First, its related results can be proved in a fast and nice way by virtue of the convex character of the functional approach. Second, its related operator version, which coincides with the previous one, can be immediately deduced without referring to the techniques of the operator mean theory. Third, as it is well known the operator mean theory has been investigated when the involved operators act on a Hilbert space. However, the functional mean theory works when the referential space is just a locally convex topological vector space $E$, and specially if $E$ is a real or complex normed vector space. In this paper, this latter point will be explored and explained in a detailed manner.\\

The present manuscript will be organized as follows: Section 2 is devoted to state some basic notions from convex analysis that will be needed throughout the next sections. Section 3 deals with the primordial and typical example of convex functional generated by a positive linear operator. In Section 4 we recall some means with functional arguments that were recently investigated in the literature. Section 5 discusses further properties about the three standard arithmetic/harmonic/geometric functional means. Section 6 displays the logarithmic mean with convex functional variables that extends the logarithmic mean of positive operators. Section 7 deals with some inequalities involving the previous logarithmic functional mean. As already pointed before, the operator versions of all functional results obtained in this paper are immediately deduced without any more tools. Otherwise, the present work highlights the importance of the convex analysis when applied to the theory of operator/functional means.

\section{\bf Background material from convex analysis}

We collect in this section some basic notions and results about the Fenchel duality in convex analysis. For more details, we refer the interested reader to \cite{AUB,EKT,HIU,LAU,ROC} for instance.

Let $E$ be a real or complex locally convex topological vector space and $E^*$ its topological dual. The notation $\langle.,.\rangle$ refers to the bracket duality between $E$ and $E^*$. Throughout the following, we set
$$\widetilde{\mathbb R}=:{\mathbb R}\cup\{+\infty\},\;\; \overline{\mathbb R}=:{\mathbb R}\cup\{-\infty,+\infty\}.$$
We also denote by $\widetilde{\mathbb R}^E$ the set of all functionals defined from $E$ into $\widetilde{\mathbb R}$.

$\bullet$ As usual in convex analysis, we extend here the structure of the field ${\mathbb R}$ to $\overline{\mathbb R}$ by setting, for any $a\in\overline{\mathbb R}$,
$$a+(+\infty)=+\infty,\; (+\infty)-(+\infty)=+\infty,\; 0.(+\infty)=+\infty,$$
and the total order of ${\mathbb R}$ is extended to $\overline{\mathbb R}$ by, $a\leq b$ if and only if $b-a\geq 0$, with the usual convention $-\infty\leq a\leq+\infty$, for any $a,b\in\overline{\mathbb R}$. We pay attention here to the fact that $a\leq b$ is not equivalent to $a-b\leq0$, by virtue of the convention $(+\infty)-(+\infty)=+\infty$.

$\bullet$ Let $f:E\longrightarrow\overline{\mathbb R}$ be a given functional. As usual, we say that $f$ is convex if $$f\big((1-t)x+ty\big)\leq(1-t)f(x)+tf(y)$$
whenever $x,y\in E$ and $t\in[0,1]$. For a subset $C$ of $E$, we denote by $\Psi_C:E\longrightarrow\widetilde{\mathbb R}$ the indicator function of $C$ defined by $\Psi_C(x)=0$ if $x\in C$ and $\Psi_C(x)=+\infty$ else. It is easy to see that the set $C$ is convex if and only if $\Psi_C$ is a convex functional. Further if $C$ is convex then, $f$ is convex on $C$ if and only if $f+\Psi_C$ is convex. By virtue of the definition of the indicator function and its properties, it is henceforth enough to consider functionals defined on the whole space $E$.

$\bullet$ We denote by $dom\;f$ the effective domain of $f$ defined by $dom\;f=\{x\in E:\; f(x)<+\infty\}$ and we say that $f$ is proper if $f$ does not take the value $-\infty$ and $f$ is not identically equal to $+\infty$. Clearly, if $f$ is proper then $dom\;f\neq\emptyset$. Further, if $f$ is a convex functional then $dom\;f$ is a convex set, but the converse is not always true. For example, if $E$ is a normed space and we take $f(x)=-\|x\|$ then $dom\;f=E$ is a convex set while $f$ is not a convex function.

$\bullet$ The notation $\Gamma_0(E)$ stands for the set of all convex lower semi-continuous (l.s.c) proper functionals defined on $E$. It is well known that $f\in\Gamma_0(E)$ if and only if the epigraph of $f$, namely $epi(f)=:\{(x,\lambda)\in E\times{\mathbb R}:\; f(x)\leq\lambda\}$, is convex and closed in $E\times{\mathbb R}$. It is not hard to see that if $C$ is a nonempty convex closed subset of $E$ then $\Psi_C$ belongs to $\Gamma_0(E)$, and vice versa. It is easy to check that $\Gamma_0(E)$ is a convex cone of $\tilde{\mathbb R}^E$. That is, if $f,g\in\Gamma_0(E)$ and $\alpha\geq0$ is a real number then $f+g\in\Gamma_0(E)$ and $\alpha.f\in\Gamma_0(E)$.

$\bullet$ Let $f:E\longrightarrow\overline{\mathbb R}$. The Fenchel conjugate (or dual) of $f$ is the functional $f^*:E^*\longrightarrow\overline{\mathbb R}$ defined by
\begin{equation*}
\forall x^*\in E^*\;\;\;\;\; f^*(x^*)=:\sup_{x\in E}\Big\{\Re e\langle x^*,x\rangle-f(x)\Big\}=\sup_{x\in {\rm dom\;f}}\Big\{\Re e\langle x^*,x\rangle-f(x)\Big\}.
\end{equation*}
It is worth mentioning that, if $f(x_0)=-\infty$ for some $x_0\in E$ then $f^*(x^*)=+\infty$ for any $x^*\in E^*$. For fixed $x\in E$, the real-maps $x^*\longmapsto\phi_{x}(x^*)=:\Re e\langle x^*,x\rangle-f(x)$ are linear affine and l.s.c and so $f^*$ is convex and l.s.c as a supremum of a family of convex and l.s.c functionals, even if $f$ is or not convex l.s.c. The following inequality, known as Fenchel inequality,
\begin{equation}\label{105}
\Re e\langle x^*,x\rangle\leq f(x)+f^*(x^*)
\end{equation}
holds for any $x\in E$ and $x^*\in E^*$. It is easy to check that $(f+c)^*=f^*-c$ for any $c\in{\mathbb R}$. Further, if for $\alpha>0$ we define $(\alpha.f)(x)=:\alpha f(x)$ and $(f.\alpha)(x)=:\alpha f(x/\alpha)$, then one can easily check that
\begin{equation}\label{110}
\big(\alpha.f\big)^*=f^*.\alpha\;\; \mbox{and}\;\; \big(f.\alpha\big)^*=\alpha.f^*.
\end{equation}
The duality map $f\longmapsto f^*$ is point-wisely increasing and convex. That is, for any $f,g\in\widetilde{\mathbb R}^E$ and $t\in[0,1]$ we have, $f\leq g\Longrightarrow g^*\leq f^*$ and
\begin{equation}\label{pc}
\Big((1-t)f+tg\Big)^*\leq(1-t)f^*+tg^*,
\end{equation}
where the notation $f\leq g$ refers to the partial point-wise order defined by: $f\leq g$ if and only if $f(x)\leq g(x)$ i.e. $g(x)-f(x)\geq0$, for any $x\in E$.

$\bullet$ For $f:E\longrightarrow\overline{\mathbb R}$ we denote by $f^{**}$ the bi-conjugate (or bi-dual) of $f$ defined from $E$ into $\overline{\mathbb R}$ by
$$\forall x\in E\;\;\;\;\;\;\; f^{**}(x)=:\sup_{x^*\in E^*}\Big\{\Re e\langle x^*,x\rangle-f^*(x^*)\Big\}.$$
It is worth mentioning that, by definition, $f^{**}$ is not the conjugate of $f^*$, and so whenever we speak for the Fenchel duality, we use only the duality $\langle.,.\rangle$ between $E$ and $E^*$ and that between $E^*$ and $E^{**}$ is omitted, unless $E$ is a reflexive Banach space i.e. $E=E^{**}$ as Banach spaces case for which the two preceding dualities coincide. The following inequality $f^{**}\leq f$ holds true. We sometimes call $f^{**}$ the convex closure of $f$ i.e. $f^{**}$ is the greatest convex l.s.c function less than $f$. Further, $f\in\Gamma_0(E)$ if and only if $f^{**}=f$.

$\bullet$ Let $x\in dom\;f$. The notation $\partial f(x)$ refers to the sub-differential of $f$ at $x$ defined by
$$\partial f(x)=:\Big\{x^*\in E^*:\; f(z)\geq f(x)+\Re e\langle x^*,z-x\rangle,\;\; \mbox{for any}\; z\in E\Big\}.$$
Note that, $\partial f(x)$ is a subset of $E^*$ always convex and closed but possibly empty. However, if the topological interior of $dom f$, denoted by $int\big(dom f\big)$, is not empty then $\partial f(x)\neq\emptyset$ for any $x\in int\big(dom f\big)$. The following equivalence
$$x^*\in\partial f(x)\Longleftrightarrow f(x)+f^*(x^*)=\Re e\langle x^*,x\rangle$$
holds true. If $x^*\in\partial f(x)$ then $x\in\partial f^*(x*)$, with reverse implication provided that $f\in\Gamma_0(E)$.

$\bullet$ Let $x\in dom\;f$ and $d\in  E$. The directional derivative of $f$ in the direction $d$ at $x$ is defined by
$$f^{'}(x,d)=:\lim_{t\downarrow0}\frac{f(x+td)-f(x)}{t},$$
provided that this limit exists. We say that $f$ is G\^{a}teaux-differentiable, in short G-differentiable, at $x$ if the map $d\longmapsto f^{'}(x,d)$ is linear continuous, i.e. $f^{'}(x,d)=\langle z^*,d\rangle$ for some $z^*\in E^*$. Such $z^*$ is unique and we set $z^*=:\nabla f(x)$ which is called the G-gradient of $f$ at $x$. If $f$ is convex and G-differentiable at $x$ then $\partial f(x)=\{\nabla f(x)\}$. Inversely, if $f$ is convex and $\partial f(x)$ is a singleton then $f$ is G-differentiable at $x$ and $\partial f(x)=\{\nabla f(x)\}$.

$\bullet$ Let $f,g\in\widetilde{\mathbb R}^E$. The inf-convolution of $f$ and $g$ is defined by
$$\forall x\in E\;\;\;\;\;\;\; f\Box g(x)=:\inf_{z\in E}\{f(z)+g(x-z)\}.$$
It is well known that $dom(f\Box g)=dom\;f+dom\;g$. Further, if $f$ and $g$ are convex then so is $f\Box g$, but $f,g\in\Gamma_0(E)$ does not imply $f\Box g\in\Gamma_0(E)$. Otherwise, the following relationship holds true
\begin{equation}\label{115}
(f\Box g)^*=f^*+g^*.
\end{equation}
However, the relationship $(f+g)^*=f^*\Box g^*$ is not always true. Note that, by \eqref{115} we have $\big(f^*+g^*\big)^*=\big(f\Box g\big)^{**}$. Therefore,
\begin{equation}\label{120}
\big(f^*+g^*\big)^*=f\Box g\;\Longleftrightarrow\;f\Box g\in\Gamma_0(E).
\end{equation}
In the literature, we can find a list of assumptions under what the condition $f\Box g\in\Gamma_0(E)$ holds. For example, if the following condition
\begin{equation}\label{125}
int\big(dom\;f^*\big)\cap dom\;g^*\neq\emptyset\;\; \mbox{or}\;\; dom\;f^*\cap int\big(dom\;g^*\big)\neq\emptyset,
\end{equation}
holds then $f\Box g\in\Gamma_0(E)$. For more information and details about this latter point, we refer the interested reader to \cite{HIU,LAU,ROC} for instance.

\section{\bf Generated function by linear operator}

In this section, we will consider a typical and interesting example of a convex functional generated by a linear operator. Here, $H$ denotes a real or complex Hilbert space. Following the Riesz representation, the bracket duality here is the inner product of $H$, also denoted by $\langle.,.\rangle$. We then denote by ${\mathcal B}(H)$ the space of all bounded linear operators defined from $H$ into itself. For $A\in{\mathcal B}(H)$, we say that $A$ is positive, and we write $A\geq0$, if $\langle Ax,x\rangle\geq0$ for any $x\in H$. The positiveness of operators generates a partial order between self-adjoint operators defined by: $A\leq B$ if and only if $A$ and $B$ are self-adjoint and $B-A\geq0$. We say that $A$ is strictly positive, and we write $A>0$, if $A$ is positive and invertible. If $H$ is a finite dimensional space then, $A$ is strictly positive if and only if $\langle Ax,x\rangle>0$ for any $x\in E$, with $x\neq0$. We denote by ${\mathcal B}^{+*}(H)$ the set of all positive invertible operators in ${\mathcal B}(H)$.

For every $A\in{\mathcal B}(H)$ we can derive a functional ${\mathcal Q}_A$ defined by
$$\forall x\in H\;\;\;\;\;\;\; {\mathcal Q}_A(x)=\frac{1}{2}\langle Ax,x\rangle,$$
which will be called the quadratic function generated by $A$. Note that, as we will see later, the coefficient $(1/2)$ appearing in ${\mathcal Q}_A$ will play a good tool for a symmetrization reason when computing the conjugate of ${\mathcal Q}_A$. It is clear that ${\mathcal Q}_I(x)=:\frac{1}{2}\|x\|^2$, where $I$ is the identity operator of $H$ and $\|.\|$ is the hilbertian norm of $H$.

The elementary properties of ${\mathcal Q}_A$ are summarized in the following proposition.

\begin{proposition}\label{prEl}
Let $A,B\in{\mathcal B}(H)$. Then the following assertions hold:\\
(i) Assume that $A$ and $B$ are self-adjoint. Then, ${\mathcal Q}_A\leq (\geq){\mathcal Q}_B$ if and only if $A\leq (\geq)B$.\\
(ii) ${\mathcal Q}_A+{\mathcal Q}_B={\mathcal Q}_{A+B}$ and $\alpha {\mathcal Q}_A={\mathcal Q}_{\alpha A}$ for any real number $\alpha$.\\
(iii) ${\mathcal Q}_A$ is continuous. Further, ${\mathcal Q}_A$ is convex if and only if $A$ is positive.\\
(iv) Assume that $A\in{\mathcal B}^{+*}$(H). Then the conjugate of ${\mathcal Q}_A$ is given by
$$\forall x^*\in E^*\;\;\;\;\;\;\; {\mathcal Q}_A^*(x^*)=\frac{1}{2}\langle A^{-1}x^*,x^*\rangle,$$
or, in short,
\begin{equation}\label{inv}
{\mathcal Q}_A^*={\mathcal Q}_{A^{-1}}.
\end{equation}
(v) ${\mathcal Q}_A$ is G-differentiable at any $x\in H$. If further $A$ is self-adjoint then $\nabla {\mathcal Q}_A(x)=Ax$. So, $\partial f(x)=\{Ax\}$ whenever $A$ is (self-adjoint) positive.\\
(vi) Let $A,B\in{\mathcal B}^{+*}$(H). Then ${\mathcal Q}_A\square{\mathcal Q}_B={\mathcal Q}_{A//B}$ where $A//B=:\big(A^{-1}+B^{-1}\big)^{-1}$ is called the parallel sum of $A$ and $B$.
\end{proposition}
\begin{proof}
The proofs of (i),(ii),(iii) and (v) are straightforward. For the proof of (iv), see \cite{RLAMA} for instance. For more details about (vi) we can consult \cite{AND1,ANT}.
\end{proof}

\begin{remark}
By \eqref{inv} we immediately deduce that ${\mathcal Q}_I^*={\mathcal Q}_I=:(1/2)\|.\|^2$. That is, $(1/2)\|.\|^2$ is self-conjugate. By using the Fenchel inequality \eqref{105} it is not hard to check that $(1/2)\|.\|^2$ is the unique self-conjugate functional defined on a Hilbert space.
\end{remark}

We have the following result as well.

\begin{theorem}\label{thF}
Let $\Phi:\Gamma_0(H)\times\Gamma_0(H)\longrightarrow\widetilde{\mathbb R}^H$ and $\Psi:\Gamma_0(H)\times\Gamma_0(H)\longrightarrow\widetilde{\mathbb R}^H$ be two binary maps such that $\Phi(f,g)\leq\Psi(f,g)$ for any $f,g\in\Gamma_0(H)$. Assume that, for any $A,B\in{\mathcal B}^{+*}(H)$ we have
$$\Phi({\mathcal Q}_A,{\mathcal Q}_B)={\mathcal Q}_{\theta(A,B)}\;\; \mbox{and}\;\; \Psi({\mathcal Q}_A,{\mathcal Q}_B)={\mathcal Q}_{\gamma(A,B)},$$
where $\theta(A,B),\gamma(A,B)\in{\mathcal B}(H)$ are self-adjoint. Then
$$\theta(A,B)\leq\gamma(A,B).$$
\end{theorem}
\begin{proof}
Since $A,B\in{\mathcal B}^{+*}(H)$ then, by Proposition \ref{prEl},(iii) we have ${\mathcal Q}_A,{\mathcal Q}_B\in\Gamma_0(H)$. It follows that $\Phi({\mathcal Q}_A,{\mathcal Q}_B)\leq\Psi({\mathcal Q}_A,{\mathcal Q}_B)$ and so ${\mathcal Q}_{\theta(A,B)}\leq {\mathcal Q}_{\gamma(A,B)}$, with $\theta(A,B)$ and $\gamma(A,B)\in{\mathcal B}(H)$ are self-adjoint. By Proposition \ref{prEl},(i) we conclude that $\theta(A,B)\leq\gamma(A,B)$ and the proof is complete.
\end{proof}

Theorem \ref{thF} is a simple and central result which will be substantially used throughout this paper. It shows how to obtain an operator inequality from an inequality involving convex functionals. The following examples give more explanation about the use of this theorem as well as the preceding properties and concepts. Further examples of interest will be seen in the next sections.

\begin{example}
Let $A,B\in{\mathcal B}^{+*}(H)$. Assume that $A\leq B$ then, by Proposition \ref{prEl},(i), we have ${\mathcal Q}_A\leq {\mathcal Q}_B$. By the point-wise decrease monotonicity of the Fenchel duality we infer that ${\mathcal Q}_B^*\leq {\mathcal Q}_A^*$ and by \eqref{inv} we deduce that ${\mathcal Q}_{B^{-1}}\leq {\mathcal Q}_{A^{-1}}$. Again by Proposition \ref{prEl},(i) we conclude that $B^{-1}\leq A^{-1}$. This means that the map $X\longmapsto X^{-1}$, for $X\in{\mathcal B}^{+*}(H)$, is operator monotone (increasing).
\end{example}
\begin{example}
For fixed $t\in[0,1]$, we set
$$\Phi(f,g)=\big((1-t)f+tg\big)^*\;\; \mbox{and}\;\; \Psi(f,g)=(1-t)f^*+tg^*.$$
Following \eqref{pc} we have
$\Phi(f,g)\leq\Psi(f,g)$ for any $f,g\in\Gamma_0(H)$. Since $A,B\in{\mathcal B}^{+*}(H)$ then by Proposition \ref{prEl},(iii) one has ${\mathcal Q}_A,{\mathcal Q}_B\in\Gamma_0(H)$. Further, by Proposition \ref{prEl},(ii) and \eqref{inv} we can write
$$\Phi({\mathcal Q}_A,{\mathcal Q}_B)={\mathcal Q}_{\theta(A,B)},\;\; \mbox{with}\;\; \theta(A,B)=\big((1-t)A+tB\big)^{-1},$$
$$\Psi({\mathcal Q}_A,{\mathcal Q}_B)={\mathcal Q}_{\gamma(A,B)},\;\; \mbox{with}\;\; \gamma(A,B)=(1-t)A^{-1}+tB^{-1}.$$
According to Theorem \ref{thF} we conclude that $\theta(A,B)\leq\gamma(A,B)$. This means that the map $X\longmapsto X^{-1}$, for $X\in{\mathcal B}^{+*}(H)$, is operator convex.
\end{example}

\section{\bf Functional means}

In this section we will recall some functional means already investigated in the literature. Throughout this section and the next ones, $E$ denotes a real or complex topological locally convex vector space, as previous, and $H$ denotes a real or complex Hilbert space.

Let $(f,g)\in \Gamma_0(E)$ and $\lambda\in(0,1)$. The following expressions, \cite{RAB}
\begin{eqnarray}
&& f\nabla_{\lambda}g=:(1-\lambda)f+\lambda g,\;  f!_{\lambda}g=:\Big((1-\lambda)f^*+\lambda g^*\Big)^*,\nonumber \\
&& f\sharp_{\lambda}g=:\displaystyle{\frac{\sin(\pi\lambda)}{\pi}\int_{0}^1\frac{t^{\lambda-1}}{(1-t)^{\lambda}}}f!_tg\;dt\label{420}
\end{eqnarray}
are known as the $\lambda$-weighted functional arithmetic mean, the $\lambda$-weighted functional harmonic mean and the $\lambda$-weighted functional geometric mean of $f$ and $g$, respectively. For $\lambda=1/2$, they are simply denoted by $f\nabla g,\; f!g$ and $f\sharp g$, respectively. For another definition of $f\sharp g$ as point-wise limit of an algorithm descending from $f\nabla g$ and $f!g$ we can consult \cite{ATR}.

\begin{remark}
(i) The $\lambda$-weighted functional geometric mean can be written as follows:
$$f\sharp_{\lambda}g=\int_0^1f!_tgd\nu_{\lambda}(t),\;\; \lambda\in(0,1),$$
where $\nu_{\lambda}(t)$ defines a family of probability measures on the interval $(0,1)$ defined by
\begin{equation}\label{425}
d\nu_{\lambda}(t)=:\frac{\sin(\pi\lambda)}{\pi}\frac{t^{\lambda-1}}{(1-t)^{\lambda}}dt,\;\; \lambda\in(0,1).
\end{equation}
(ii) Although the previous functional means can be defined, by the same expressions, even $f,g\notin\Gamma_0(E)$, we restrict ourselves throughout this paper to assume that $f,g\in\Gamma_0(E)$. In this case, $f\nabla_{\lambda}g,f!_{\lambda}g$ and $f\sharp_{\lambda}g$ belong to $\Gamma_0(E)$ provided that $dom\;f\cap dom\;g\neq\emptyset$.
\end{remark}

We extend the previous functional means on the whole interval $[0,1]$ by setting:
\begin{equation}\label{430}
f\nabla_0g=f!_0g=f\sharp_0g=f\;\; \mbox{and}\;\; f\nabla_1g=f!_1g=f\sharp_1g=g.
\end{equation}
We pay attention that these latter relations can not be deduced from \eqref{420}, by virtue of the convention $0.(+\infty)=+\infty$. The previous functional means satisfy the following relationships
\begin{equation}\label{435}
f\nabla_{\lambda}g=g\nabla_{1-\lambda}f,\; f!_{\lambda}g=g!_{1-\lambda}f,\; f\sharp_{\lambda}g=g\sharp_{1-\lambda}f,
\end{equation}
for any $\lambda\in[0,1]$. The two first relationships of \eqref{435} are immediate and for the third one we can consult \cite{RAB}. In particular, if $\lambda=1/2$, the three previous functional means are symmetric in $f$ and $g$. Note that, $f\nabla_{\lambda}f=f!_{\lambda}f=f\sharp_{\lambda}f=f$. Further, the following inequalities hold, see \cite{RAB}
\begin{equation}\label{440}
f!_{\lambda}g\leq f\sharp_{\lambda}g\leq f\nabla_{\lambda}g,
\end{equation}
and $f\nabla_{\lambda}g, f!_{\lambda}g,f\sharp_{\lambda}g\in\Gamma_0(E)$ provided that $dom\;f\cap dom\;g\neq\emptyset$. Denoting by $\it{m}_\lambda$ one of any mean among $\nabla_\lambda,\;!_\lambda,\;\sharp_\lambda$ and utilizing \eqref{110}, we can easily see that, for any $\alpha>0$,
\begin{equation}\label{445}
\alpha.f\it{m}_\lambda\alpha.g=\alpha.\big(f\it{m}_\lambda g\big),\;\; \mbox{and}\;\; f.\alpha\it{m}_\lambda g.\alpha=\big(f\it{m}_\lambda g\big).\alpha
\end{equation}

Otherwise, for any $A,B\in{\mathcal B}^{+*}(H)$, we have the following relationships
\begin{equation}\label{450}
{\mathcal Q}_A\nabla_{\lambda}{\mathcal Q}_B={\mathcal Q}_{A\nabla_{\lambda}B},\; {\mathcal Q}_A!_{\lambda}{\mathcal Q}_B={\mathcal Q}_{A!_{\lambda}B},\; {\mathcal Q}_A\sharp_{\lambda}{\mathcal Q}_B={\mathcal Q}_{A\sharp_{\lambda}B},
\end{equation}
where
\begin{multline}\label{455}
A\nabla_{\lambda}B=:(1-\lambda)A+\lambda B,\; A!_{\lambda}B=:\Big((1-\lambda)A^{-1}+\lambda B^{-1}\Big)^{-1},\\
A\sharp_{\lambda}B=:A^{1/2}\Big(A^{-1/2}BA^{-1/2}\Big)^{\lambda}A^{1/2}
\end{multline}
stands for the $\lambda$-weighted operator arithmetic mean, the $\lambda$-weighted operator harmonic mean and the $\lambda$-weighted operator geometric mean of $A$ and $B$, respectively. For $\lambda=1/2$, they are also simply denoted by $A\nabla B,\; A!B$ and $A\sharp B$, respectively. The relationships \eqref{450} justify that the previous functional means are, respectively, extensions of their related operator means. Further, according to Theorem \ref{thF}, \eqref{440} and \eqref{450} immediately imply that the following operator inequalities
\begin{equation}\label{460}
A!_{\lambda}B\leq A\sharp_{\lambda}B\leq A\nabla_{\lambda}B
\end{equation}
hold for any $A,B\in{\mathcal B}^{+*}(H)$ and $\lambda\in[0,1]$. It is worth mentioning that \eqref{460}, which are well-known in the operator mean theory, are here again obtained in a simultaneous manner and under a convex point of view that does not need to refer to the techniques of functional calculus.

\section{\bf More properties for $f\nabla_{\lambda}g,\;f!_{\lambda}g,\;f\sharp_{\lambda}g$}

In the ongoing section, we will be interested by studying other properties of the functional means $f\nabla_{\lambda}g,\;f!_{\lambda}g$ and $f\sharp_{\lambda}g$. First, let $A,B\in{\mathcal B}^{+*}(H)$. It is well known that $A\nabla_\lambda B,A!_\lambda B$ and $A\sharp_\lambda B$ are monotone increasing with respect to both $A$ and $B$ \cite{Hiai2010}. Otherwise, obviously the map $(A,B)\longmapsto A\nabla_\lambda B$ is linear affine while $(A,B)\longmapsto A!_\lambda B$ is operator concave, see \cite{AND}, and so $(A,B)\longmapsto A\sharp_\lambda B$ is also operator concave. In what follows we will present the extensions of these latter operator properties for convex functionals. It is clear that $(f,g)\longmapsto f\nabla_\lambda g$ is linear affine and point-wisely increasing with respect to $f$ and $g$. We now state the following result.

\begin{proposition}\label{prPM}
Let $f,g\in\Gamma_0(E)$ and $\lambda\in[0,1]$. Then the two binary maps $(f,g)\longmapsto f!_{\lambda}g$ and $(f,g)\longmapsto f\sharp_{\lambda}g$ are both separately point-wisely increasing.
\end{proposition}
\begin{proof}
Let $f_1,f_2\in\Gamma_0(E)$ be such that $f_1\leq f_2$. By the point-wise decrease monotonicity of the map $f\longmapsto f^*$, we deduce $(1-\lambda)f_2^*+\lambda g^*\leq (1-\lambda)f_1^*+\lambda g^*$ and again $\big((1-\lambda)f_1^*+\lambda g^*\big)^*
\leq\big((1-\lambda)f_2^*+\lambda g^*\big)^*$ i.e. $f_1!_{\lambda}g\leq f_2!_{\lambda}g$ which mean that $(f,g)\longmapsto f!_{\lambda}g$ is point-wisely increasing with respect to the first argument $f$. By virtue of \eqref{435} we then deduce that $(f,g)\longmapsto f!_{\lambda}g$ is point-wisely increasing with respect to the second argument $g$, too. This, with the relation of $f\sharp_{\lambda}g$ given in \eqref{420} and the linearity of the integral, implies that $(f,g)\longmapsto f\sharp_{\lambda}g$ is separately point-wisely increasing. The proof is complete.
\end{proof}

In order to give another result of interest, we need to introduce the following notation:
$${\mathcal W}=:\Big\{(f,g)\in\Gamma_0(E)\times\Gamma_0(E):\; dom\;f^*=E^*\; \mbox{and}\; dom\;g^*=E^*\Big\}.$$
Obviously, ${\mathcal W}$ is a cone, with $(f,g)\in{\mathcal W}$ if and only if $(g,f)\in{\mathcal W}$. Note that ${\mathcal Q}_A,{\mathcal Q}_B\in{\mathcal W}$ for any $A,B\in{\mathcal B}^{+*}(H)$. Further, it is easy to check that ${\mathcal W}$ is convex i.e. if $(f_1,g_1)\in{\mathcal W}$ and $(f_2,g_2)\in{\mathcal W}$ then $\big((1-t)f_1+tf_2,(1-t)g_1+tg_2\big)\in{\mathcal W}$ for any $t\in(0,1)$. We now state the following result.

\begin{theorem}\label{thPC}
Let $(f,g)\in{\mathcal W}$ and $\lambda\in[0,1]$. Then the two binary maps $(f,g)\longmapsto f!_{\lambda}g$ and $(f,g)\longmapsto f\sharp_{\lambda}g$ are both point-wisely concave.
\end{theorem}
\begin{proof}
By the same reasons as in the proof of the previous proposition, we need to prove that $(f,g)\longmapsto f!_{\lambda}g$ is point-wisely concave with respect to the first argument $f$. By definition of $f!_\lambda g$, with the help of \eqref{110} and \eqref{115}, we can write
\begin{equation}\label{470}
f!_\lambda g=\big((1-\lambda)f^*+\lambda g^*\big)^*=\big(f.(1-\lambda)\Box g.\lambda\big)^{**}.
\end{equation}
Since $(f,g)\in{\mathcal W}$ then it is easy to verify that the condition \eqref{125} is here satisfied i.e. $int\big(dom(1-\lambda).f^*\big)\cap dom\big(\lambda.g^*\big)\neq\emptyset$. It follows that $f.(1-\lambda)\Box g.\lambda\in\Gamma_0(E)$ and so \eqref{470} becomes
\begin{equation}\label{472}
f!_{\lambda}g=f.(1-\lambda)\Box g.\lambda.
\end{equation}
Now, let $(f_1,g_1),(f_2,g_2)\in{\mathcal W}$ and $t\in(0,1)$. By \eqref{472} and the definition of the inf-convolution, we have for any $x\in E$
\begin{multline*}
\big((1-t)f_1+tf_2\big)!_{\lambda}\big((1-t)g_1+tg_2\big)(x)\\
=\inf_{z\in E}\Big\{\big((1-t)f_1+tf_2\big).(1-\lambda)(z)+
\big((1-t)g_1+tg_2\big).\lambda(x-z)\Big\}\\
=\inf_{z\in E}\left\{(1-t)\Big(f_1.(1-\lambda)(z)+g_1.\lambda(x-z)\Big)+t\Big(f_2.(1-\lambda)(z)+g_2.\lambda(x-z)\Big)\right\}\\
\geq(1-t)\inf_{z\in E}\Big(f_1.(1-\lambda)(z)+g_1.\lambda(x-z)\Big)+t\inf_{z\in E}\Big(f_2.(1-\lambda)(z)+g_2.\lambda(x-z)\Big)\\
=(1-t)\big(f_1.(1-\lambda)\Box g_1.\lambda\big)(x)+t\big(f_2.(1-\lambda)\Box g_2.\lambda\big)(x)=(1-t)f_1!_\lambda g_1(x)+tf_2!_\lambda g_2(x).
\end{multline*}
Hence the desired result.
\end{proof}

\begin{remark}
Let $(f,g)\in{\mathcal W}$. Then, for any $\lambda\in(0,1)$ we have
\begin{equation}\label{475}
dom\;f\cap dom\;g\subset dom\big(f\sharp_{\lambda}g\big)\subset(1-\lambda)dom\;f+\lambda\;dom\;g.
\end{equation}
Indeed, it is easy to see that $dom(f\nabla_\lambda g)=dom\;f\cap dom\;g$ for any $\lambda\in(0,1)$. Otherwise, by using \eqref{472} it is not hard to check that
$$dom(f!_{\lambda}g)=(1-\lambda)dom\;f+\lambda\;dom\;g.$$
This, when combined with \eqref{440}, yields \eqref{475}.
\end{remark}

Now remark that, for $f,g\in\Gamma_0(E)$ fixed, the map $t\longmapsto f\nabla_tg$ is point-wisely affine (so convex and concave). Otherwise, we have the following result.

\begin{proposition}\label{prchm}
Let $f,g\in\Gamma_0(E)$ be fixed. Then the map $t\longmapsto f!_tg$ is point-wisely convex on $[0,1]$.
\end{proposition}
\begin{proof}
By definition we have $f!_tg=\big((1-t)f^*+tg^*\big)^*$. Since the map $\phi\longmapsto \phi^*$ is point-wisely convex and the map $t\longmapsto(1-t)f^*+tg^*$ is point-wisely affine, the desired result follows immediately.
\end{proof}

Now, we will construct a family of functional means which enjoys interesting properties. Let $f,g\in\Gamma_0(E)$ and $\lambda\in[0,1]$ be fixed. For $s\in[0,1]$ we set
\begin{equation}\label{480}
G_s\big(f,g;\lambda\big)=\int_0^1f!_{st+(1-s)\lambda}g\;d\nu_{\lambda}(t),
\end{equation}
where $d\nu_{\lambda}(t)$ is defined by \eqref{425}. Observe that $G_s(f,f;\lambda)=f$ for any $s,\lambda\in[0,1]$ and $f\in\Gamma_0(E)$. The family $G_s(f,g;\lambda)$, when $s$ describes the interval $[0,1]$, includes the functional means $f!_\lambda g$ and $f\sharp_{\lambda}g$ in the sense that $G_0(f,g;\lambda)=f!_\lambda g$ and $G_1(f,g;\lambda)=f\sharp_{\lambda}g$. The basic properties of the maps $s\longmapsto G_s(f,g;\lambda)$ are encapsulated in the following result.

\begin{theorem}\label{thG}
With the above, the following assertions are met:\\
(i) The map $s\longmapsto G_s(f,g;\lambda)$ is point-wisely convex on $[0,1]$.\\
(ii) For any $s\in[0,1]$, we have
\begin{equation}\label{485}
f!_\lambda g\leq G_s\big(f,g;\lambda\big)\leq (f!_\lambda g)\nabla_s(f\sharp_\lambda g)\leq f\sharp_{\lambda}g\;\big(\leq f\nabla_{\lambda}g\big),
\end{equation}
which refines the left inequality in \eqref{440}.\\
(iii) We have
\begin{equation}\label{487}
\inf_{s\in[0,1]}G_s\big(f,g;\lambda\big)=f!_\lambda g\;\; \mbox{and}\;\; \sup_{s\in[0,1]}G_s\big(f,g;\lambda\big)=f\sharp_{\lambda}g,
\end{equation}
where the infimum and supremum are taken for the point-wise order.\\
(iv) The map $s\longmapsto G_s(f,g;\lambda)$ is point-wisely monotone increasing.
\end{theorem}
\begin{proof}
(i) Since the map $t\longmapsto f!_tg$ is point-wisely convex on $[0,1]$ and the real-function $s\longmapsto st+(1-s)\lambda\in[0,1]$ is affine then we deduce the desired result.\\
(ii) By the point-wise convexity of $t\longmapsto f!_tg$ we get
$$f!_{st+(1-s)\lambda}g\leq sf!_tg+(1-s)f!_{\lambda}g.$$
If we multiply this latter inequality by $d\nu_\lambda(t)$ and we integrate over $t\in[0,1]$ we get the middle inequality of \eqref{485}. The right inequality of \eqref{485} is obvious by virtue of the inequality $f!_{\lambda}g\leq f\sharp_\lambda g$. Now, let us show the left inequality of \eqref{485}. For the sake of simplicity for the reader, we fix $f,g\in\Gamma_0(E)$ and we set $\Phi(s)=f!_sg$. By Proposition \ref{prchm}, $\Phi$ is point-wisely convex on $[0,1]$. With this, \eqref{480} takes the following form
$$G_s(f,g;\lambda)=\int_0^1\Phi(st+(1-s)\lambda)\;d\nu_\lambda(t).$$
We can apply the integral Jensen inequality to this latter equality \cite{DRR}, and we then obtain
\begin{equation}\label{490}
G_s(f,g;\lambda)\geq\Phi\Big(\int_0^1\big(st+(1-s)\lambda\big)\;d\nu_\lambda(t)\Big).
\end{equation}
We have
\begin{multline*}
\int_0^1\big(st+(1-s)\lambda\big)\;d\nu_\lambda(t)\\=s\int_0^1t\;d\nu_\lambda(t)+(1-s)\lambda\int_0^1\;d\nu_\lambda(t)
=s\int_0^1t\;d\nu_\lambda(t)+(1-s)\lambda.
\end{multline*}
In another part, let us denote by $\Gamma$ and $B$ the standard special gamma and beta functions, respectively. By \eqref{425} we get
\begin{multline*}
\int_0^1t\;d\nu_\lambda(t)=\frac{\sin(\pi\lambda)}{\pi}\int_0^1t^{\lambda}(1-t)^{-\lambda}dt
=\frac{\sin(\pi\lambda)}{\pi}B(1+\lambda,1-\lambda)\\
=\frac{\sin(\pi\lambda)}{\pi}\frac{\Gamma(1+\lambda)\Gamma(1-\lambda)}{\Gamma(2)}=\frac{\sin(\pi\lambda)}{\pi}\;\lambda\;\Gamma(\lambda)\Gamma(1-\lambda)=\lambda.
\end{multline*}
Substituting these in \eqref{490} we obtain
$$G_s(f,g;\lambda)\geq\Phi(s\lambda+(1-s)\lambda)=\Phi(\lambda)=:f!_{\lambda}g,$$
and the left inequality of \eqref{485} is obtained.\\
(iii) Since $G_0(f,g;\lambda)=f!_\lambda g$ and $G_1(f,g;\lambda)=f\sharp_{\lambda}g$, then \eqref{487} are immediate from \eqref{485}.\\
(iv) If $0\leq s_1<s_2\leq1$, the point-wise convexity of $s\longmapsto G_s(f,g;\lambda)$ implies that
$$\frac{G_{s_2}(f,g;\lambda)-G_{s_1}(f,g;\lambda)}{s_2-s_1}\geq\frac{G_{s_1}(f,g;\lambda)-G_{0}(f,g;\lambda)}{s_1}.$$
This, with $G_0(f,g;\lambda)=f!_\lambda g$ and \eqref{487}, yields the desired result. The proof is complete.
\end{proof}

\section{\bf Logarithmic mean of two convex functionals}

In this section we will introduce a logarithmic mean of two convex functionals.
Let $f,g\in\Gamma_0(E)$. As pointed out in \cite{RAB}, the map $t\longmapsto f\sharp_tg$ is point-wisely continuous on $(0,1)$. We can then put the following.

\begin{definition}\label{def51}
Let $f,g\in\Gamma_0(E)$. The expression
\begin{equation}\label{510}
L(f,g)=:\int_0^1f\sharp_tgdt
\end{equation}
is called the logarithmic mean of $f$ and $g$.
\end{definition}

The terminology used in the preceding definition will be justified later. The basic properties of $L(f,g)$ are embodied in the following result.

\begin{proposition}
Let $f,g\in\Gamma_0(E)$. Then the following assertions hold:\\
(i) $L(f,f)=f$ and $L(f,g)=L(g,f)$.\\
(ii) For any $c,d\in{\mathbb R}$, $L(f+c,g+d)=L(f,g)+c\nabla d$ where $c\nabla d=:\frac{c+d}{2}$ denotes the extension of the arithmetic mean for any two real numbers.\\
(iii) Let $\alpha,\beta>0$, then
$$L(\alpha.f,\beta.g)=(\alpha\sharp\beta).L(\alpha\sharp\beta^{-1}.f,\alpha^{-1}\sharp\beta.g),$$
$$L(f.\alpha,g.\beta)=L(f.\alpha\sharp\beta^{-1},g.\alpha^{-1}\sharp\beta).(\alpha\sharp\beta),$$
where $\alpha\sharp\beta=:\sqrt{\alpha\beta}$ is the scalar geometric mean of $\alpha$ and $\beta$. In particular, one has
$$L(\alpha.f,\alpha.g)=\alpha.L(f,g)\;\; \mbox{and}\;\; L(f.\alpha,g.\alpha)=L(f,g).\alpha$$
\end{proposition}
\begin{proof}
(i) Since $f\sharp_tf=f$ for any $t\in[0,1]$ then \eqref{510} gives $L(f,f)=f$. Making the change of variables $t=1-u$ in \eqref{510} we deduce
$L(f,g)=\int_0^1f\sharp_{1-u}gdu$. This with \eqref{435} implies $L(f,g)=L(g,f)$.\\
(ii) and (iii) According to \eqref{110}, with some basic operations and manipulations, we immediately deduce the desired equalities. The details are straightforward and therefore omitted here.
\end{proof}

The following result contains more interesting properties of $L(f,g)$.

\begin{theorem}
Let $f,g\in\Gamma_0(E)$. Then \\
(i) If $(f,g)\in{\mathcal W}$ then the binary map $(f,g)\longmapsto L(f,g)$ is separately point-wise increasing and separately point-wise concave.\\
(ii) The functional arithmetic-logarithmic-harmonic inequality holds, i.e.
\begin{equation}\label{513}
f!g\leq L(f,g)\leq f\nabla g.
\end{equation}
Thus, $L(f,g)\in\Gamma_0(E)$ provided that $dom\;f\cap dom\;g\neq\emptyset$.
\end{theorem}
\begin{proof}
(i) follows from Proposition \ref{prPM} and Theorem \ref{thPC}, with \eqref{510}, respectively.\\
(ii) By the right inequality in \eqref{440}, we have $f\sharp_tg\leq f\nabla_tg$ for all $t\in[0,1]$. Integrating this latter inequality over $t\in[0,1]$, and noting that $\int_0^1f\nabla_tgdt=f\nabla g$ we get the right inequality of \eqref{513}. Now, we prove the left inequality of \eqref{513}. First, recall that we have $f!_tg=\big(f^*\nabla_tg^*\big)^*$. In another part, by \eqref{510} with the left inequality of \eqref{440} we get
$$L(f,g)=\int_0^1f\sharp_tg\;dt\geq\int_0^1f!_tg\;dt=\int_0^1\big(f^*\nabla_tg^*\big)^*dt.$$
According to \eqref{pc}, the duality map $\phi\longmapsto \phi^*$ is point-wisely convex. Following \cite{DRR}, such map satisfies the integral Jensen inequality. Thus, we obtain
$$L(f,g)\geq\Big(\int_0^1f^*\nabla_tg^*dt\Big)^*=\big(f^*\nabla g^*\big)^*=:f!g.$$
From \eqref{510} we deduce that $L(f,g)$ is convex l.s.c, as integral of a family of convex l.s.c functionals. Further, from the right inequality in \eqref{513} we infer that $L(f,g)$ is proper whenever $f\nabla g$ is, i.e. $dom\;f\cap dom\;g\neq\emptyset$. The proof is completed.
\end{proof}

The following result is as well of interest and justifies the terminology used in Definition \ref{def51}.

\begin{proposition}
Let $A,B\in{\mathcal B}^{+*}(H)$. Then we have
\begin{equation}\label{517}
L({\mathcal Q}_A,{\mathcal Q}_B)={\mathcal Q}_{L(A,B)},
\end{equation}
where ${\mathcal Q}_A$ refers to the generated function of $A$ previously defined, and
\begin{equation}\label{519}
L(A,B)=:\int_0^1A\sharp_tBdt=A^{1/2}F\big(A^{-1/2}BA^{-1/2}\big)A^{1/2},
\end{equation}
with $F(x)=\frac{x-1}{\log x}$ for $x>0,\; x\neq1$, and $F(1)=1$. That is, $L(A,B)$ is the operator logarithmic mean of $A$ and $B$.
\end{proposition}
\begin{proof}
By \eqref{510}, with the third relation in \eqref{450}, we can write
$$L\big({\mathcal Q}_A,{\mathcal Q}_B\big)=\int_0^1{\mathcal Q}_A\sharp_t{\mathcal Q}_Bdt=\int_0^1{\mathcal Q}_{A\sharp_tB}dt
={\mathcal Q}_{\int_0^1A\sharp_tBdt}=:{\mathcal Q}_{L(A,B)},$$
with
\begin{multline*}
L(A,B)=\int_0^1A\sharp_tBdt=A^{1/2}\int_0^1\Big(A^{-1/2}BA^{-1/2}\Big)^tdt\;A^{1/2}\\=:A^{1/2}F\big(A^{-1/2}BA^{-1/2}\big)A^{1/2},
\end{multline*}
where, for $x>0$,
$$F(x)=:\int_0^1x^tdt=\Big[\frac{x^t}{\log x}\Big]_0^1=\frac{x-1}{\log x}\;\; \mbox{if}\;\; x\neq1,\;\; F(1)=1.$$
The proof is complete.
\end{proof}

In order to state another result of interest, we need the following lemma.

\begin{lemma}
Let $x>0$ then we have
\begin{equation}\label{525}
\phi(x)=:\int_0^1x^v\sin(\pi v)dv=\frac{(x+1)\pi}{\pi^2+\big(\log x\big)^2}.
\end{equation}
\end{lemma}
\begin{proof}
We consider $\psi(x)=:\int_0^1x^v\cos(\pi v)dv$ and we compute $\psi(x)+i\phi(x)$, where $i^2=-1$. Elementary computation of integral leads to
$$\psi(x)+i\phi(x)=\int_0^1x^ve^{i\pi v}dv=\int_0^1\exp v\big(i\pi+\log x\big)dv=\frac{-x-1}{i\pi+\log x}.$$
Separating real part and imaginary part, we get the desired result, so completing the proof.
\end{proof}

In Definition \ref{def51}, $L(f,g)$ is defined in terms of the weighted geometric functional mean. The following result gives an expression of $L(f,g)$ in terms of the weighted harmonic functional mean.

\begin{theorem}\label{thFM}
Let $f,g\in\Gamma_0(E)$. Then
\begin{equation}\label{530}
L(f,g)=\int_0^1f!_tg\;d\mu(t),
\end{equation}
where $\mu(t)$ is the probability measure on $(0,1)$ defined by
\begin{equation}\label{535}
d\mu(t)=:\frac{dt}{t(1-t)\Big(\pi^2+\big(\log\frac{t}{1-t}\big)^2\Big)}.
\end{equation}
\end{theorem}
\begin{proof}
By \eqref{510} and the third relation in \eqref{420} we can write
$$L(f,g)=\int_0^1f\sharp_sg\;ds=\int_0^1\frac{\sin(\pi s)}{\pi}\int_0^1\frac{t^{s-1}}{(1-t)^s}f!_tg\;dt\;ds,$$
or, equivalently,
$$L(f,g)=\frac{1}{\pi}\int_0^1\frac{f!_tg}{t}\int_0^1\sin(\pi s)\Big(\frac{t}{1-t}\Big)^s\;ds\;dt.$$
This, with \eqref{525} and a simple reduction, yields
$$L(f,g)=\int_0^1\frac{f!_tg\;dt}{t(1-t)\Big(\pi^2+\big(\log\frac{t}{1-t}\big)^2\Big)}=:\int_0^1f!_tg\;d\mu(t).$$
The fact that $L(f,f)=f$ and $f!_tf=f$ for any $t\in(0,1)$ and $f\in\Gamma_0(E)$ implies that $\int_0^1d\mu(t)=1$ i.e. $\mu(t)$ is a probability measure on $(0,1)$. The proof is complete.
\end{proof}

The preceding theorem, when interpreted in terms of functional mean and operator mean, brings us some information of great interest. In fact, first mention that the operator version of \eqref{530} is given in the following.

\begin{corollary}
For any $A,B\in{\mathcal B}^{+*}(H)$ we have
\begin{equation}\label{540}
L(A,B)=\int_0^1A!_tB\;d\mu(t),
\end{equation}
where $L(A,B)$ is given by \eqref{519} and $\mu(t)$ is defined in Theorem \ref{thFM}.
\end{corollary}
\begin{proof}
Taking $f={\mathcal Q}_A$ and $g={\mathcal Q}_B$ in \eqref{530}, with the help of \eqref{517}, we get
$$L\big({\mathcal Q}_A,{\mathcal Q}_B\big)=\int_0^1{\mathcal Q}_A!_t{\mathcal Q}_Bd\mu(t)=\int_0^1{\mathcal Q}_{A!_tB}d\mu(t)={\mathcal Q}_{\int_0^1A!_tBd\mu(t)}={\mathcal Q}_{L(A,B)}.$$
The desired result follows.
\end{proof}

Now, let us observe the following remark which explains some interesting topics.

\begin{remark}
The Kubo-Ando theory for monotone operator means, \cite{KA}, tells us that every operator mean $A\sigma B$ can be written in the following form
\begin{equation}\label{545}
A\sigma B=\int_0^1A!_tB\;dp(t),
\end{equation}
where $p(t)$ is a certain probability measure on $(0,1)$ depending on the operator mean $\sigma$. This, when combined with \eqref{540}, gives us the explicit probability $\mu(t)$ on $(0,1)$ corresponding to the logarithmic operator mean $L(A,B)$. Further, combining \eqref{530} and \eqref{545} we can infer that $L(f,g)$, previously defined, is a reasonable extension of $L(A,B)$ when the positive operator variables $A$ and $B$ are replaced by convex functional arguments $f$ and $g$, respectively.
\end{remark}

The following remark may be of interest as well.

\begin{remark}
Theorem \ref{thFM} may be good tool again for bringing us some information about computation of some integrals that are not simple to compute directly. Indeed, the fact that $d\mu(t)$ is a probability measure on $(0,1)$, with a simple decomposition, yields
\begin{equation}\label{547}
\int_0^1\frac{dt}{t\Big(\pi^2+\big(\log\frac{t}{1-t}\big)^2\Big)}=\int_0^1\frac{dt}{(1-t)\Big(\pi^2+\big(\log\frac{t}{1-t}\big)^2\Big)}=\frac{1}{2}.
\end{equation}
By simple change of variables in these latter integrals, we get
$$\int_0^{\pi/2}\frac{\tan z\;dz}{\pi^2+4\big(\log(\tan z)\big)^2}=\int_0^{\pi/2}\frac{\cot z\;dz}{\pi^2+4\big(\log(\cot z)\big)^2}
=\frac{1}{4},$$
and then
$$\int_0^\infty\frac{u\;du}{(1+u^2)\big(\pi^2+(\log u)^2\big)}=\frac{1}{4}.$$
\end{remark}

Another remark which gives more explanation about the interest of \eqref{530} is recited in what follows.

\begin{remark}
Let $\Psi:(0,1)\longrightarrow{\mathbb R}$ be defined by:
\begin{equation}\label{550}
\forall t\in(0,1)\;\;\;\;\;\;\; \Psi(t)=\frac{1}{t(1-t)\Big(\pi^2+\big(\log\frac{t}{1-t}\big)^2\Big)}.
\end{equation}
It is easy to see that $\Psi$ is a symmetric density function on $(0,1)$ i.e. $\Psi(t)\geq0$ and $\Psi(t)=\Psi(1-t)$ for any $t\in(0,1)$, and $\int_0^1\Psi(t)dt=1$. Since $t\longmapsto f!_tg$ is point-wisely convex, see Proposition \ref{prchm}, we can then apply the so-called F\'{e}jer-Hermite-Hadamard inequality, \cite{DRP}, to \eqref{530}. In fact, for any $x\in E$, we have
$$f!_{1/2}g(x)\leq\int_0^1\Psi(t)(f!_tg)(x)dt\leq\frac{f!_0g(x)+f!_1g(x)}{2}.$$
This, with $f!_0g=f,\; f!_1g=g,\;f!_{1/2}g=f!g$ and \eqref{530}, immediately yields again \eqref{513}.
\end{remark}

\section{\bf Further inequalities about $L(f,g)$}

In this section we will give more inequalities involving $L(f,g)$. We first state the following result which will be needed in the sequel, see \cite{RAI2}.

\begin{lemma}
Let $f,g\in\Gamma_0(E)$. For each $t,s\in(0,1)$, we have
\begin{equation}\label{610}
0\leq r_{t,s}\big(f\nabla_sg-f!_sg\big)\leq f\nabla_tg-f!_tg\leq R_{t,s}\big(f\nabla_sg-f!_sg\big),
\end{equation}
where we set
$$r_{t,s}=:\min\Big(\frac{t}{s},\frac{1-t}{1-s}\Big)\;\; \mbox{and}\;\; R_{t,s}=:\max\Big(\frac{t}{s},\frac{1-t}{1-s}\Big).$$
\end{lemma}

Now, we are in the position to state the following result which concerns a refinement and reverse for the right inequality in \eqref{513}.

\begin{theorem}\label{thRR}
Let $f,g\in\Gamma_0(E)$. Then for any $s\in(0,1)$ we have
\begin{equation}\label{612}
0\leq\Big(\frac{1}{2}-I_{s}\Big)\frac{f\nabla_sg-f!_sg}{s(1-s)}
\leq f\nabla g-L(f,g)\leq I_s\frac{f\nabla_sg-f!_sg}{s(1-s)},
\end{equation}
where we set
$$I_s=:s\int_0^s\omega(t)dt+(1-s)\int_0^{1-s}\omega(t)dt,\;\;\omega(t)=:\frac{1}{t\Big(\pi^2+\big(\log\frac{t}{1-t}\big)^2\Big)},\; t\in(0,1).$$
\end{theorem}
\begin{proof}
Using \eqref{547}, it is easy to check that $f\nabla g=\displaystyle\int_0^1f\nabla_tg\;d\mu(t)$, where $\mu(t)$ is defined in Theorem \ref{thFM}. It follows that
$$f\nabla g-L(f,g)=\int_0^1\big(f\nabla_tg-f!_tg\big)d\mu(t).$$
This, with \eqref{610}, yields
$$a_s\big(f\nabla_sg-f!_sg\big)\leq f\nabla g-L(f,g)\leq b_s\big(f\nabla_sg-f!_sg\big),$$
where we set
$$a_s=:\int_0^1r_{t,s}\;d\mu(t),\;\; b_s=:\int_0^1R_{t,s}\;d\mu(t).$$
Let us remark that $R_{t,s}=\frac{t}{s}$ if $t\geq s$ and $R_{t,s}=\frac{1-t}{1-s}$ if $t\leq s$. By the elementary techniques of integration, it is not hard to check that
$$b_s=\frac{1}{1-s}\int_0^s\omega(t)dt+\frac{1}{s}\int_0^{1-s}\omega(t)dt=\frac{1}{s(1-s)}I_s.$$
For computing $a_s$ we use the fact that
$$r_{t,s}+R_{t,s}=\frac{t}{s}+\frac{1-t}{1-s}=\frac{t(1-s)+s(1-t)}{s(1-s)}.$$
Multiplying by $d\mu(t)$ and integrating over $t\in(0,1)$, with the help of \eqref{547}, we get $a_s+b_s=\frac{1}{2s(1-s)}$. Otherwise, by \eqref{547} again we can write
$$I_s\leq s\int_0^1\omega(t)dt+(1-s)\int_0^1\omega(t)dt=\frac{1}{2}.$$
Summarizing, we get \eqref{612}, so completing the proof.
\end{proof}

Remark that \eqref{612} implies that $1/4\leq I_s\leq1/2$ for any $s\in(0,1)$. Taking $s=1/2$ in \eqref{612} we immediately obtain the following result which refines the right inequality of \eqref{513}.

\begin{corollary}\label{corRR}
Let $f,g\in\Gamma_0(E)$. Then one has
$$0\leq(2-I)\big(f\nabla g-f!g\big)\leq f\nabla g-L(f,g),$$
or, equivalently, as an upper bound of $L(f,g)$ in convex combination of $f\nabla g$ and $f!g$
$$L(f,g)\leq(I-1)f\nabla g+(2-I)f!g\leq f\nabla g,$$
where we set
$$I=:4\;I_{1/2}=4\int_0^{1/2}\omega(t)dt,\;\; 1\leq I\leq 2.$$
\end{corollary}

Note that the operator versions of Theorem \ref{thRR} and Corollary \ref{corRR} are immediate. Otherwise, the exact value of the integral $I_{1/2}=\int_0^{1/2}\omega(t)dt\in[1/4,1/2]$ seems to be uncomputable.

Corollary \ref{corRR} gives an upper bound of $L(f,g)$. For giving a lower bound of $L(f.g)$ we need the following lemma.

\begin{lemma}\label{lemRAA}
Let $f,g\in\Gamma_0(E)$ and $t\in(0,1)$. Let $x\in E$ be such that $\partial f(x)\neq\emptyset$ and $\partial g(x)\neq\emptyset$. Then, for any $x^*\in\partial f(x)$ and $z^*\in\partial g(x)$, we have the following inequality
\begin{equation}\label{613}
\big(0\leq\big)f\nabla_tg-f\sharp_tg\leq t(1-t)\Big({\mathcal F}_g(x,x^*)\nabla{\mathcal F}_f(x,z^*)\Big),
\end{equation}
where, for any $y\in E$, $y^*\in E^*$ and $h:E\longrightarrow{\mathbb R}\cup\{-\infty,+\infty\}$, we set
\begin{equation}\label{6135}
{\mathcal F}_h(y,y^*)=:h(y)+h^*(y^*)-\Re e\langle y^*,y\rangle\geq0.
\end{equation}
\end{lemma}
\begin{proof}
This result was proved in \cite{RAA} when $E$ is a Hilbert space. The same proof works when $E$ is an arbitrary locally convex topological space.
\end{proof}

For $f,g\in\Gamma_0(E)$, we also need to introduce the following notation
\begin{equation}\label{614}
f\diamond g(x)=:\sup_{x^*\in\partial f(x)}\Big\{\Re e\langle x^*,x\rangle-g^*(x^*)\Big\},
\end{equation}
with the usual convention $\sup_{\emptyset}(...)=-\infty$. The elementary properties of the law $(f,g)\longmapsto f\diamond g$ are summarized in the following result.

\begin{proposition}\label{prdiamond}
Let $f,g\in\Gamma_0(E)$. The following assertions hold:\\
(i) For any $x\in E$, we have
$$f\diamond g(x)=\big(g^*+\Psi_{\partial f(x)}\big)^*(x).$$
(ii) Let $A,B\in{\mathcal B}^{+*}(H)$. Then one has
$${\mathcal Q}_A\diamond{\mathcal Q}_B={\mathcal Q}_{A\diamond B},\;\; \mbox{with}\;\; A\diamond B=:2A-AB^{-1}A.$$
(iii) $f\diamond g$ is not always convex.\\
(iv) $f\diamond g\leq g$  and so, $B-A\diamond B$ is positive for any $A,B\in{\mathcal B}^{+*}(H)$.
\end{proposition}
\begin{proof}
(i) Follows from \eqref{614}, with the definition of the conjugate duality.\\
(ii) We use \eqref{inv} and Proposition \ref{prEl},(v) with some algebraic manipulations. The details are straightforward and therefore omitted here.\\
(iii) It follows from (ii), since $2A-AB^{-1}A$ is not always positive.\\
(iv) It is a consequence of (i) and (ii).
\end{proof}

Before stating another main result, we mention the following remark which is of interest.

\begin{remark}
Since our involved functionals can take the values $\pm\infty$ we then have to be careful with certain critical situations. If fact, the equality $f-f=0$ is not always true. Precisely, we have $f-f=\Psi_{dom\;f}$ by virtue of the convention $(+\infty)-(+\infty)=+\infty$. For the same reason, the equality $f-g=-(g-f)$ is not always true. Also, the inequality $f\leq g$ is equivalent to $g-f\geq0$ but it is not equivalent to $f-g\leq0$.
\end{remark}

We are now in the position to state the following result which reverses the right inequality in \eqref{513}.

\begin{theorem}\label{thD}
Let $f,g\in\Gamma_0(E)$. Let $x\in E$ be such that $\partial f(x)\neq\emptyset$ and $\partial g(x)\neq\emptyset$. Then, for any $x^*\in\partial f(x)$ and $z^*\in\partial g(x)$ we have
\begin{equation}\label{615}
0\leq f\nabla g(x)-L(f,g)(x)\leq\frac{1}{6}\Big({\mathcal F}_g(x,x^*)\nabla{\mathcal F}_f(x,z^*)\Big).
\end{equation}
Or, equivalently,
\begin{equation}\label{616}
0\leq f\nabla g-L(f,g)\leq\frac{1}{6}\Big(f\nabla g-(f\diamond g)\nabla(g\diamond f)\Big).
\end{equation}
\end{theorem}
\begin{proof}
Integrating \eqref{613} side by side over $t\in[0,1]$, with \eqref{510} and the fact that $\int_0^1f\nabla_tg\;dt=f\nabla g$, we get \eqref{615}. We now prove \eqref{616}. According to the previous remark, we begin by discussing some typical situations. Let $x\in E$. If $\partial f(x)=\emptyset$ or $\partial g(x)=\emptyset$ then by \eqref{613} we infer that $f\diamond g(x)=-\infty$ or $g\diamond f(x)=-\infty$, respectively. In this case we have $(f\diamond g)\nabla(g\diamond f)(x)\in\{-\infty,+\infty\}$. If $(f\diamond g)\nabla(g\diamond f)(x)=+\infty$, then by the first part of Proposition \ref{prdiamond},(iv) we deduce that $f(x)=+\infty$ or $g(x)=+\infty$ and so all sides of \eqref{615} take the value $+\infty$ at $x$. If $(f\diamond g)\nabla(g\diamond f)(x)=-\infty$ thus the two right sides of \eqref{615} are both equal to $+\infty$, by virtue of the convention $(c)-(-\infty)=+\infty$ for any $c\in{\mathbb R}\cup\{-\infty,+\infty\}$. It follows that \eqref{615} is satisfied at the point $x$, since $c\leq+\infty$ for any $c\in{\mathbb R}\cup\{-\infty,+\infty\}$. Now, assume that $\partial f(x)\neq\emptyset$ and $\partial g(x)\neq\emptyset$. Then \eqref{615} is satisfied at $x$ and it is then equivalent to the following inequality
\begin{equation}\label{618}
0\leq f\nabla g(x)-L(f,g)(x)\leq\frac{1}{6}\left(\inf_{x^*\in\partial f(x)}{\mathcal F}_g(x,x^*)\nabla\inf_{z^*\in\partial g(x)}{\mathcal F}_f(x,z^*)\right).
\end{equation}
By \eqref{6135} and \eqref{614} we have
$$\inf_{x^*\in\partial f(x)}{\mathcal F}_g(x,x^*)=g(x)-f\diamond g(x),\;\; \inf_{z^*\in\partial g(x)}{\mathcal F}_f(x,z^*)=f(x)-g\diamond f(x).$$
Substituting this in \eqref{618} we get the right inequality of \eqref{616} at $x$. The proof is complete.
\end{proof}

By using Proposition \ref{prdiamond},(ii) we left to the reader the routine task to check that the operator version of Theorem \ref{thD} may be recited as follows.

\begin{corollary}
For any $A,B\in{\mathcal B}^{+*}(H)$ there holds
\begin{equation}\label{620}
0\leq A\nabla B-L(A,B)\leq\frac{1}{6}\Big(\big(AB^{-1}A\big)\nabla\big(BA^{-1}B\big)-A\nabla B\Big).
\end{equation}
\end{corollary}

\begin{remark}
It is easy to check that the scalar version of \eqref{620} reads as follows: for any real numbers $a,b>0$ we have
$$0\leq a\nabla b-L(a,b)\leq\frac{2}{3}\big(a\nabla b\big)\Big((a\nabla b)^2-(a\sharp b)^2\Big),$$
where $L(a,b)$ is the standard logarithmic mean of $a$ and $b$ i.e. $L(a,b)=:\frac{a-b}{\log\;a-\log\;b}$, for $a\neq b$, and $L(a,a)=a$.
\end{remark}

Now, we will be interested by refining the left inequality of \eqref{513}. Let $f,g\in\Gamma_0(E)$ be fixed. For $s\in[0,1]$, we set
\begin{equation}\label{622}
U_s(f,g)=\int_0^1\Big(f!_{st+(1-s)\frac{1}{2}}g\Big)d\mu(t),
\end{equation}
where $d\mu(t)$ is defined by \eqref{535}. Remark that $U_s(f,f)=f$ for any $f\in\Gamma_0(E)$ and $s\in[0,1]$. The map $s\longmapsto U_s(f,g)$ enjoys nice properties which we embody in the following result.

\begin{theorem}\label{thU}
The following assertions hold true:\\
(i) The map $s\longmapsto U_s(f,g)$ is point-wisely convex on $[0,1]$.\\
(ii) For any $s\in[0,1]$, we have the inequalities
\begin{equation}\label{625}
f!g\leq U_s(f,g)\leq (f!g)\nabla_sL(f,g)\leq L(f,g),
\end{equation}
which refines the left inequality in \eqref{513}.\\
(iii) We have
\begin{equation}\label{627}
\inf_{s\in[0,1]}U_s(f,g)=f!g\;\; \mbox{and}\;\; \sup_{s\in[0,1]}U_s(f,g)=L(f,g),
\end{equation}
where the infimum and supremum are taken for the point-wise order.\\
(iv) The map $s\longmapsto U_s(f,g)$ is point-wisely monotone increasing.
\end{theorem}
\begin{proof}
(i) By Proposition \ref{prchm}, with the fact that $s\longmapsto st+(1-s)/2$ is affine for fixed $t\in[0,1]$, we deduce that, for any $t\in[0,1]$, the family of maps $s\longmapsto f!_{st+(1-s)/2}g$ is point-wisely convex on $[0,1]$. Thus $s\longmapsto U_s(f,g)$ is also point-wisely convex on $[0,1]$.\\
(ii) Since $t\longmapsto f!_tg$ is point-wisely convex then we can write
$$f!_{st+(1-s)/2}g\leq sf!_tg+(1-s)f!_{1/2}g.$$
Multiplying this latter inequality by $d\mu(t)$ and integrating over $t\in[0,1]$ we obtain the middle inequality in \eqref{625}, since $f!_{1/2}g=f!g$. The right inequality in \eqref{625} is immediate, since $f!g\leq L(f,g)$. Now, we will prove the left inequality in \eqref{625}. As for the proof of Theorem \ref{thG}, we fix $f,g\in\Gamma_0(E)$ and we simply set $\Phi(s)=f!_sg$. By Proposition \ref{prchm}, $\Phi$ is point-wisely convex on $[0,1]$. In another part, \eqref{622} can be written as follows
$$U_s(f,g)=\int_0^1\Phi(st+(1-s)/2)\;d\mu(t).$$
Writing this equality point-wisely we can then use the integral Jensen inequality, see also \cite{DRR}, and we get
\begin{equation}\label{630}
U_s(f,g)\geq\Phi\left(\int_0^1\left(st+(1-s)/2\right)\;d\mu(t)\right).
\end{equation}
We have, by utilizing \eqref{535} and \eqref{547},
\begin{multline*}
\int_0^1\left(st+(1-s)/2\right)\;d\mu(t)\\=s\int_0^1t\;d\mu(t)+(1-s)/2\int_0^1\;d\mu(t)
=s/2+(1-s)/2=1/2.
\end{multline*}
Substituting this in \eqref{630} we obtain
$$U_s(f,g)\geq\Phi(1/2)=:f!_{1/2}g=:f!g,$$
whence the left inequality in \eqref{625}.\\
(iii) By \eqref{622}, it is clear that $U_0(f,g)=f!g$ and $U_1(f,g)=L(f,g)$. This, with \eqref{625}, implies \eqref{627}.\\
(iv) Let $s_1,s_2\in[0,1]$ be such that $s_1<s_2$. Since $s\longmapsto U_s(f,g)$ is point-wisely convex then we have
$$\frac{U_{s_2}(f,g)-U_{s_1}(f,g)}{s_2-s_1}\geq\frac{U_{s_1}(f,g)-U_{0}(f,g)}{s_1}.$$
By (iii) we have $U_{s_1}(f,g)-U_{0}(f,g)\geq0$ for any $s_1\in[0,1]$, since $U_0(f,g)=f!g$. Hence the desired result, so completing the proof.
\end{proof}

\begin{remark}
We left to the reader the task for formulating in an immediate way the analog of Theorem \ref{thU} when the two convex functionals $f$ and $g$ are replaced by two positive invertible operators $A$ and $B$, respectively.
\end{remark}

\section*{Acknowledgements}
The authors would like to thank the referees for their careful and insightful comments to improve our manuscript.
The author (S.F.) was partially supported by JSPS KAKENHI Grant Number 16K05257.


\begin{thebibliography} {9}


\bibitem{FM2020} S. Furuichi and H. R.Moradi, \textit{Advances in Mathematical Inequalities}, De Gruyter, 2020.
\bibitem{BMV} P. S. Bullen, D. S. Mitrinovi\'{c}, and P. M. Vasi\'{c}, \textit{Means and their Inequalities}, Reidel, Dordrecht, 1988.
\bibitem{KA} F. Kubo and T. Ando, \textit{Means of positive linear operators}, Math. Ann. \textbf{246} (1980), 205--224.
\bibitem{NUC} R. D. Nussbaum and J. E. Cohen, \textit{The arithmetic-geometric mean and its generalizations for noncommuting linear operators}, Ann. Sci. Norm. Sup. Sci(4) \textbf{15/2} (1989), 239--308.
\bibitem{ATR} M. Atteia and M. Ra\"{\i}ssouli, \textit{Self dual operators on convex functionals, geometric mean and  square root of convex functionals}, J. Conv. Anal. \textbf{8/1} (2001), 223--240.
\bibitem{FUJ} J. I. Fujii, \textit{Kubo-Ando theory for convex functional means}, Sci. Math. Japonicae \textbf{7} (2002), 299--311.
\bibitem{RAB} M. Ra\"{\i}ssouli and H. Bouziane, \textit{Arithmetico-geometrico-harmonic functional mean in convex analysis}, Ann. Sc. Math. Qu\'ebec \textbf{30/1} (2006), 79--107.
\bibitem{RAC} M. Ra\"{\i}ssouli and M. Chergui, \textit{Arithmetico-geometric and geometrico-harmonic means of two convex functionals}, Sci. Math. Japonicae \textbf{55/3} (2002), 485--492.
\bibitem{RAI1} M. Ra\"{\i}ssouli, \textit{Logarithmic functional mean in convex analysis}, J. Ineq. Pure Appl. Math. \textbf{10}(2009), no. 4, Art. 102.
\bibitem{RAI2} M. Ra\"{\i}ssouli, \textit{Functional versions of some refined and reversed operator mean-inequalities}, Bull. Austr. Math. Soc. \textbf{96} (2017), no. 3, 496--503, DOI: 10.1017/S0004972717000594.
\bibitem{RAF} M. Ra\"{\i}ssouli and S. Furuichi, \textit{Some inequalities involving Heron and Heinz means of two convex functionals}, Analysis Math. \textbf{46/2} (2020), 345--365, DOI: 10.1007/s10476-020-0026-x.
\bibitem{AUB} J. P. Aubin, \textit{Analyse non lin\'eaire et ses motivations \'economiques}, Masson, 1983.
\bibitem{EKT} I. Ekeland and R. Temam, \textit{Convex Analysis and Variational Problems}, SIAM, 1999.
\bibitem{HIU} J. B. Hiriart-Urruty, \textit{$\epsilon$-subdifferential calculus in convex analysis and optimization}, Research notes in Mathematics, 57, Pitman Publishers, 1982.
\bibitem{LAU} P. J. Laurent, \textit{Approximation et optimisation}, Hermann, 1972.
\bibitem{ROC} R. T. Rockafellar, \textit{Convex Analysis}, Princeton University Press, New Jersey, 1970.
\bibitem{RLAMA} M. Ra\"{\i}ssouli, \textit{Some inequalities involving quadratic forms of operator means}, Linear and Multilinear Algebra \textbf{67} (2019), no. 2, 213--220, DOI: 10.1080/03081087.2017.1416573.
\bibitem{AND1} W. N. JR. Anderson, \textit{Shorted operators}, SIAM J. Appl. Math. \textbf{20} (1971), 520--525.
\bibitem{ANT} W. N. JR. Anderson and G. E. Trapp, \textit{Shorted operators II}, SIAM J. Appl. Math. \textbf{28} (1975), 60--71.
\bibitem{Hiai2010} F. Hiai, {\it Matrix Analysis: Matrix Monotone Functions, Matrix Means, and Majorization}, Interdisciplinary Information Sciences, {\bf 16}(2) (2010), 139--248.
\bibitem{AND} T. Ando, \textit{Concavity of certain maps on positive definite matrices and applications to Hadamard products}, Linear Algebra Appl. \textbf{26} (1979), 203--241.
\bibitem{DRR} S. S. Dragomir and M. Ra\"{\i}ssouli, \textit{Hermite-Hadamard inequalities for point-wise convex maps and Legendre-Fenchel conjugation}, Math. Ineq. Appl. \textbf{16/1} (2013), 143--152.
\bibitem{DRP} S. S. Dragomir and C. E. M. Pearce, \textit{Selected topics on Hermite-Hadamard inequalities and applications}, RGMIA Monograph, 2002.
\bibitem{RAA} M. Ra\"{\i}ssouli and M. AlMozini, \textit{Refining and reversing the weighted arithmetic-geometric mean inequality involving convex functionals and application for the functional entropy}, J. Inequal. Appl. \textbf{2020} (2020), 92, DOI: 10.1186/s13660-020-02355-3.




\end{thebibliography}
\end{document}